%% file: GMCD-NLN.tex
\documentclass[11pt,leqno]{article}
\usepackage{graphicx, amsfonts, amsthm, amsxtra, amssymb, amscd, verbatim, rotating}
\usepackage{multirow}
\usepackage[mathscr]{euscript}
\usepackage{hyperref}
\newcommand{\cJ}{{\cal J}}

\newcommand{\cM}{{\cal M}}

\newcommand{\cO}{{\cal O}}

\newcommand{\cR}{{\cal R}}

\begin{document}

\input{notations.tex}
\newtheorem{theo}{Theorem}
\newtheorem{exam}{Example}
\newtheorem{coro}{Corollary}
\newtheorem{defi}{Definition}
\newtheorem{prob}{Problem}
\newtheorem{lemm}{Lemma}
\newtheorem{prop}{Proposition}
\newtheorem{rem}{Remark}
\newtheorem{conj}{Conjecture}
\newtheorem{calc}{}

\begin{center}
{\LARGE\bf 
Calabi-Yau modular forms in limit: Elliptic Fibrations
\footnote{ 
Math. classification:  14N35, 14J15, 32G20
\\
Keywords: modular forms, Hodge filtration, Picard-Fuchs system}}
\\
\vspace{.25in} {\large {\sc Babak Haghighat, Hossein Movasati, Shing-Tung Yau}}
\\

\end{center}

\begin{abstract}
We study the limit of Calabi-Yau modular forms, and in particular, those resulting in classical modular forms. We then study two parameter families of elliptically fibred Calabi-Yau fourfolds and describe the modular forms  arising from the degeneracy loci. In the case of  elliptically fibred Calabi-Yau threefolds our approach gives a mathematical proof of many observations about modularity properties of topological string amplitudes starting with the work of Candelas, Font, Katz and Morrison. In the case of Calabi-Yau fourfolds we derive new identities not computed before. 
\end{abstract}

\def\CYM{{\sf M}}
\def\PF{{\cal L}}
\def\en{{n}}

\section{Introduction}
Theoretical Physics and in particular string theory has provided mathematicians with many $q$-expansions
which at first glance look like modular forms. This is actually the case for some examples of such $q$-expansions, however, 
in general they transcend the world of modular and automorphic forms. The case of the mirror quintic is the most well-known one, 
and it is argued in \cite{ho22, GMCD-MQCY3, HosseinMurad} that there is a parallel modular form theory in this case. These are called 
Calabi-Yau modular forms.  In this paper we gather further evidence that Calabi-Yau modular forms are natural generalizations of 
classical automorphic forms.  It is a well-known fact that some automorphic forms are the limit of others. 
We would like to study these phenomena in the context of Calabi-Yau modular forms for the case of elliptically fibred Calabi-Yau 
manifolds. Here, as first observed in \cite{Candelas:1994hw}, the corresponding limit for many examples are modular forms for 
$\SL 2\Z$. This observation has ultimately led to a reformulation of the topological string partition function for this 
class of Calabi-Yau manifolds in terms of meromorphic Jacobi forms which has culminated in the first all-genus results 
for the Gromov-Witten theory of compact versions of these manifolds \cite{Huang:2015sta}. In the case of compact elliptically 
fibred Calabi-Yau fourfolds, which are the focus of the present paper, Gromov-Witten invariants have been computed up to 
genus one \cite{Klemm:2007in} which is the highest non-vanishing genus for fourfolds. However, a reformulation of the 
generating functions for these invariants in terms of classical modular forms is still lacking. One goal of the present paper 
is to remedy this gap by expressing  generating functions for the genus zero Gromov-Witten invariants in terms of 
$\SL 2\Z$ modular forms.
 In the case of non-rigid Calabi-Yau manifolds of dimension $\geq 3$ we do not have 
 an underlying Hermitian symmetric domain and so we have to rephrase our problem in terms of Picard-Fuchs systems. 
 Below we describe the general setting together with the main statement of our results and elaborate on 
 the motivation from Physics and in particular string theory.

\subsection{Main statement}
We start with two parameter families of elliptically fibred Calabi-Yau $n$-folds  $X_z,\ z\in(\C^2,0)$. These are 
constructed in the framework of toric geometry, see \S\ref{toricgeometry}. 
For the construction of the field of Calabi-Yau modular forms, one can skip such geometric considerations, and
one can start with the corresponding Picard-Fuchs system:
\begin{eqnarray}
\label{eq:pfoperator1}
& &L_1:=-\en\cdot \theta_1\theta_2+ \theta_1^2-a_0z_1(\theta_1+a_1)(\theta_1+a_2) =0,  \\  
\label{eq:pfoperator2}
& &L_2:=\theta_2^\en-(-1)^nz_2(\en\cdot\theta_2-\theta_1)(\en\cdot\theta_2-\theta_1+1)\cdots (\en\cdot\theta_2-\theta_1+\en-1) = 0, 
\end{eqnarray}
where $\en, a_0,a_1,a_2$ are parameters of the system. The relevant cases to String Theory are the cases $n=3,4$ and 
$(a_0,a_1,a_2)$ as in the Table \ref{taghiirnadejanam}.
If we define $\L_\en\subset [z_1,z_2,\theta_1,\theta_2]$, with $\theta_i=z_i\partial_{z_i}$, 
to be the differential left ideal  generated by the operators $L_1$ and $L_2$,  then $\PF_{\en}$ anihilates the periods 
of a $(n,0)$-forms $\omega^{(n,0)}$ in $X_{z}$. 
The system $\PF_\en$ has one holomorphic solution $\Pi^0=O(1)$ and logarithmic solutions $\Pi^a=\Pi^0\log(z_a)+O(1),\ \ a=1,2$.  
The field $\CYM_\en$ of  differential Calabi-Yau  modular forms in these  situations is the field extension $\CYM_\en$ of  $\C$ 
generated by
\begin{equation}
\label{2sept2015}
 z_1,z_2,\ \  \theta_1^i\theta_2^j\Pi^0, \ \ \theta_1^i\theta_2^j
 \left(\Pi^0\theta_a\Pi^b-\Pi^b\theta_a\Pi^0\right),\ 
\end{equation}
$$
\ a,b=1,2,\ldots, \hn:=h^{12}(X_z)=2,\ \ i,j\in\N_0.
$$
One talks about the field of Calabi-Yau modular forms 
because constructing a graded algebra in this case, similar to the algebra of modular forms, 
demands a more elaborate analysis which is beyond the scope of this work. 
We refer for a discussion on these issues for the case of the mirror quintic to \cite{ho22, GMCD-MQCY3}. 
The field $\CYM_\en$ is finitely generated, for instance, 
it is shown in \cite{HosseinMurad} that for $\en=3$ one actually needs only $\frac{3\hn^2+7\hn+4}{2}=15$ elements in the
list \eqref{2sept2015} in order to generate $\CYM_\en$. 
The modular expressions of the elements of $\CYM_\en$  are obtained after inserting 
the mirror map $(\tau_1,\tau_2)=({\frac{\Pi^1}{\Pi^0}}, {\frac{\Pi^2}{\Pi^0}} )$ or using the
$(q_1,q_2)=(e^{\tau_1},e^{\tau_2})$ coordinates. 
From now on we will use the same name for an element $f(x)$ of $\CYM_\en$ when working with different coordinate systems $x=(z_1,z_2),(\tau_1,\tau_2)$ or $ (q_1,q_2)$. 
The main result of the present paper is the following
\begin{theo}
\label{maintheo}
 Let $f(q_1,q_2)\in \CYM_\en$ and assume that it is of the form
 $$
 f=f_0(q_1)+f_1(q_1)(q_2q_1^{\frac{\en}{2}})+\cdots +f_i(q_1)(q_2 q_1^{\frac{\en}{2}})^i+\cdots
 $$
 Then for arbitrary $\en$ and $(a_0, a_1,a_2)$ as in Table  
 \ref{taghiirnadejanam}  all $f_i(e^{\tau_1})$'s are in the field of quasi-modular forms on the upper half plane $\tau_1\in \uhp$ 
 for the  subgroup of $\SL 2\Z$ listed in the same table. 
 \end{theo}
 \begin{table}
\centering
\begin{tabular}{|c|c|c|}
\hline
$(a_0,a_1,a_2)$  & Group & Modular forms \\ \hline
$(432,5/6,1/6)$ & $\SL 2\Z$ & $E_4(\tau),E_6(\tau)$ \\ \hline
$(64,3/4,1/4)$  & $\Gamma_0(2)$  & $E_2(\tau)-2E_2(2\tau),\ E_4(\tau)$ \\ \hline
$(27,2/3,1/3)$  &  $\Gamma_0(3)$ &  $E_2(\tau)-3E_2(3\tau),\ E_4(\tau), E_6(\tau)$  \\ \hline
$(16,1/2,1/2)$  &   $\Gamma(2)$ &  $\theta_2^4,\theta_3^4$  \\
\hline
\end{tabular}
\caption{Modular groups}
\label{taghiirnadejanam}
\end{table}
 There is a tremendous amount of computation in the Physics literature confirming our main theorem for $\en=3$ and
$(a_1,a_2)=(1/6,5/6)$, see \S\ref{motivationsection}.
It does not seem to us that there is any Physics for $\en\geq 5$. The case $\en=4,\  (a_1,a_2)=(1/6,5/6)$ 
is the main motivation for us. In this case we have a collection of four-point functions 
$C^{(1,1,1,1)}_{abcd}\in \CYM_\en, a,b,c,d=1,2$ which are invarinat under index permutations, for definitions see 
\eqref{eq:yukawa} and \eqref{eq:yukdef}.   For instance, 
we derive the following identity for the four-point function 
\begin{eqnarray}
	C^{(1,1,1,1)}_{2222} & = & - q_2 \left(\frac{q_1^2}{\eta^{48}}\right) \left[\frac{5}{9} E_4 E_6(35 E_4^3+37 E_6^2)\right] \nonumber \\
	~ & ~ & - q_2^2 \left(\frac{q_1^4}{\eta^{96}}\right) \Big[\frac{5}{124416}E_4 E_6 (12377569 E_4^9 + 1960000 E_2 E_4^7 E_6 \nonumber \\
	~ & ~ & + 85433141 E_4^6 E_6^2 + 4144000 E_2 E_4^4 E_6^3 + 86392307 E_4^3 E_6^4 + 2190400 E_2 E_4 E_6^5 \nonumber \\
	~ & ~ & +11544823 E_6^6)\Big] + \mathcal{O}(q_2^3). 
\end{eqnarray}
There is no enumerative geometry attached to this function. However, if we write it in terms
of three-point functions $C^{(1,1,2)}_{ab\gamma_i}, a,b=1, i=1,2$
\begin{equation} 
\label{23oct2015-1}
	C^{(1,1,1,1)}_{abcd} =-4C^{(1,1,2)}_{ab\gamma_1}C^{(1,1,2)}_{cd \gamma_1}+
	C^{(1,1,2)}_{ab\gamma_2}C^{(1,1,2)}_{cd \gamma_1}+C^{(1,1,2)}_{ab\gamma_1}C^{(1,1,2)}_{cd \gamma_2},
\end{equation}
then from $C^{(1,1,2)}_{ab\gamma_i}$ we can derive the Gromov-Witten potentials $F^0(\gamma_i),\ \ i=1,2$:
\begin{equation}
\label{23oct2015-2}
	C^{(1,1,2)}_{ab\gamma_i} = \partial_{\tau_a} \partial_{\tau_b} F^0(\gamma_i),\ \ a,b, i=1,2.
\end{equation}
We find that \eqref{23oct2015-1} together with \eqref{23oct2015-2} allows us to solve for the functions 
$C^{(1,1,2)}_{22\gamma_i}$ at least to low orders in an expansion in $q_2$ which determines the potantials 
$F^0(\gamma_i)$ in such an expansion as follows\footnote{In all following appearances of $F^0(\gamma_i)$ we will 
suppress terms logarithmic in the $q_i$ as these only contain information about classical intersection numbers.}
\begin{eqnarray}
\label{23/10/2015}
	F^0(\gamma_1) & = & - q_2 \left(\frac{q_1^2}{\eta^{48}}\right) \left[\frac{5}{18} E_4 E_6(35 E_4^3+37 E_6^2)\right] + \mathcal{O}(q_2^2), \\ \label{eq:Fgamma2}
	F^0(\gamma_2) & = & 1 + q_2 \left(\frac{q_1^2}{\eta^{48}}\right)\left[\frac{5}{10368}(10321 E_4^6 + 1680 (-24+E_2)E_4^4 E_6\right. \\
	~ & ~ &  + 59182 E_4^3 E_6^2 + 1776 (-24 + E_2)E_4 E_6^3 + 9985 E_6^4)\Big] + \mathcal{O}(q_2^2) .\nonumber 
\end{eqnarray}
We now explain the enumerative geometry of the coeffiecients of
\begin{eqnarray*}
f_1 &=&  - \frac{5}{18} \frac{1}{\eta^{48}}\left ( E_4 E_6(35 E_4^3+37 E_6^2)\right)\\
&=& 
-20q_1^{-2} + 7680q_1^{-1} - 1800000 + 278394880q_1 +  \cdots+N_{0, d_1,1}(\gamma_1)q_1^{d_1-2}+ \cdots
\end{eqnarray*}
For further details see \cite{Klemm:2007in}. The $B$-model Calabi-Yau fourfold $X_z$ underlying the Picard-Fuchs system $\L_\en,\ \ \en=4$, is mirror dual to a Calabi-Yau fourfold $\widetilde X$ which is the resolution  of the degree $24$ hypersurface in 
$\Pn5(1,1,1,1,8,12)$. The resolution is done by blowing-up once at the unique singular point $x_1=x_2=x_3=x_4=0$.
Let $\tilde D_1\cong \Pn 3$ be the 
corresponding exceptional divisor. The variety $\widetilde X$ 
has the Hodge numbers $h^{0,0}=h^{4,0}=1, \ h^{11}=2,\ h^{31}=3878,\ 
h^{22}=15564$ and its elliptic fibration is given by $\widetilde X\to \Pn 3$ which is a projection to the first four coordinates.
 Let $D_2$ be the divisor in $\widetilde X$ which is a pull-back of a linear $\Pn 2\subset \Pn 3$ and $D_1=4D_2+\tilde D_1$. 
For $\beta\in H_2(\widetilde X,\Z)$ and $\gamma\in H^4(\widetilde X,\Z)$ we have the Gromov-Witten invariants
\begin{equation}
\label{29oct2015}
N_{g,\beta}(\gamma)=\int_{[\bar M_{g,1}(\widetilde X,\beta)]^{\rm virt}} {\rm ev}^*(\gamma),
\end{equation}
where $\bar M_{g,1}(\widetilde X,\beta)$ is the moduli space of genus $g$, $1$-pointed stable maps to $\widetilde X$ representing the class
$\beta$ and ${\rm ev }:  M_{g,1}(\widetilde X,\beta)\to \widetilde X $ is the evaluation map. 
We take a basis $[E],[\Pn 1]\in H_2(\widetilde X,\Z)$, where $[E]$ is the homology class of fibers of $\widetilde X\to \Pn 3$ and 
$[\Pn 1]$ is the homology class of a line $\Pn 1$ inside $\tilde D_1$. 
We write $N_{g,d_1,d_2}(\gamma):=N_{g,d_1[E]+d_2[\Pn 1]}(\gamma)$. 
In our formula \eqref{23/10/2015}, $\gamma_1$ is the Poincar\'e dual to $D_2^2$ and $\gamma_2$ is dual to a linear 
combination of $D_2^2$ and $D_1 D_2$.  Our modular expressions for the  Gromov-Witten generating functions are proved by using the B-model
side of mirror symmetry, and showing such statements for the $A$-model side by using the definition 
\eqref{29oct2015} are highly non-trivial open problems. 

 


\subsection{Motivation}
\label{motivationsection}

Recently, there has been a lot of progress and activity in solving the topological string on elliptic Calabi-Yau threefolds 
\cite{Klemm:2012sx,Alim:2012ss,Haghighat:2013gba,Haghighat:2013tka,Haghighat:2014pva,Kim:2014dza,Cai:2014vka,Haghighat:2014vxa,
Huang:2015sta,Gadde:2015tra,Kim:2015fxa}. 
In the case of non-compact Calabi-Yau three-folds these results lead to the computation of refined stable pair invariants 
\cite{Choi:2012jz} and translate on the physics side to partition functions of 6d SCFTs. 
In the compact case the topological string partition function is the generating function of Gromov-Witten invariants and on the physics side leads to the computation of the entropy of black holes \cite{Haghighat:2015ega}. In all these cases one can observe that topological string free energies are fully expressible in terms of classical modular forms. We review these results here where we confine ourselves to the case of compact Calabi-Yau three-folds $\widetilde X$ with a complex two-dimensional base $B$ and elliptic fibre $E$\footnote{In general the fibre can degenerate over co-dimension one loci in the base and lead to more cohomology classes whose intersection matrices are given by ADE dynkin diagrams as described by Kodaira. Here we limit ourselves to the case where there is only one such cohomology class.}.
 Here one can define a generating function for the Gromov-Witten invariants in terms of a genus expansion in a parameter $\lambda$
\begin{equation}
	F(\lambda,\underline{q}) = \sum_{g=0}^{\infty} \lambda^{2g-2} F^{(g)}(\underline{q}),
\end{equation}
where the upper index $g$ indicates the genus. According to the split of the cohomology $H_2(\widetilde X,\mathbb{Z})$ into the base and the fibre cohomology, we define $q_B^{\beta} = \prod_{k=1}^{b_2(B)}\exp(2\pi i \int_{\beta} i \omega + b)$, where $\beta \in H_2(B,\mathbb{Z})$, and $q = \exp(2\pi i \int_f i\omega + b)$ with $f$ being the curve representing the fibre\footnote{$i\omega + b$ denotes the complexified K\"ahler form.}. We now define 
\begin{equation}
	F^{(g)}_{\beta}(q) = \textrm{Coeff}(F^{(g)}(\underline{q}),q_B^{\beta}).
\end{equation}
Then one observes \cite{Klemm:2012sx} that the $F_{\beta}^{(g)}(q)$ have distinguished modular properties and can be written as
\begin{equation}
	F^{(g)}_{\beta} = \left(\frac{q^{\frac{1}{24}}}{\eta}\right)^{12 \sum_i c_i \beta^i} P_{2g+6\sum_i c_i \beta^i -2}(E_2,E_4,E_6),
\end{equation}
where $P_{2g+6 \sum_i c_i \beta^i -2}(E_2,E_4,E_6)$ are (quasi)-modular forms of weight $2g+6 \sum_i c_i \beta^i - 2$ and the $c_i$ are integer coefficients depending on the base $B$. 

As was first observed in \cite{Haghighat:2013gba}, the above modularity properties can be repackaged in the topological string partition function leading to a sum over meromorphic Jacobi forms:
\begin{equation}
	Z(\underline{q},\lambda) = \exp\left(F(\lambda,\underline{q})\right) = \sum_{\beta} q_B^{\beta} Z_{\beta}(q,\lambda),
\end{equation}
where $Z_{\beta}$ are Jacobi forms of weight zero with index a quadratic form on $H_2(B,\mathbb{Z})$. This repackaging has led to the first all-genus solutions of the topological string on compact Calabi-Yau manifolds \cite{Huang:2015sta}. 

Motivated by these results, our objectives for the present paper are to give mathematical proofs for 
modularity properties of topological 
string amplitudes for elliptic Calabi-Yau n-folds with $n \geq 3$.

\subsection*{Acknowledgements}

We are thankful to Stefan Reiter who helped us regarding many properties of linear differential equations. Thanks go to Viktor Levandovskyy who thought us how to use the libraries {\tt nctools.lib} and {\tt dmodapp.lib} for non-commutative rings  in {\tt Singular}. We would also like to thank Cumrun Vafa for valuable discussions.

\section{Preliminaries}
\subsection{Toric geometry of elliptically fibred Calabi-Yau varieties}
\label{toricgeometry}
In this paper we confine ourselves to the class of elliptically fibred Calabi-Yau n-folds over $\mathbb{P}^{n-1}$. The elliptic fibre can be one of four types depending on the weighted projective space in which it is realized. Denote by $\mathbb{P}^2(w_1,\cdots, w_r)$ a projective bundle over the base $B = \mathbb{P}^{n-1}$. The four classes are given by four choices of weights $(w_1,\cdots, w_r) = \left\{(1,2,3),(1,1,2),(1,1,1),(1,1,1,1)\right\}$ leading to elliptic curves which are hypersurfaces in the first three cases and a complete intersection in the 
last case. The Calabi-Yau manifolds corresponding to the first three cases can be realized as hypersurfaces in toric ambient spaces. The corresponding polyhedron with the Mori cone vectors is given by \cite{Klemm:2012sx}:
\begin{equation}
	\begin{array}{c|cccccc|cc}
		~     & ~ & ~ & ~         & ~ & ~ & ~ & l^{(1)}            & l^{(2)} \\
		D_0 & 1 & 0 & \cdots & 0 & 0 & 0 & \sum_i e_i -1 & 0 \\
		D_1 & 1 & ~ & ~ & ~ & e_1 & e_2 & 0 & 1 \\
		\vdots & 1 & ~ & \Delta_B & ~ & \vdots & \vdots & \vdots & \vdots \\
		D_n & 1 & ~  & ~ & ~ & e_1 & e_2 & 0 & 1 \\
		D_z & 1 & 0 & \cdots & 0 & e_1 & e_2 & 1 & -n \\
		D_x & 1 & 0 & \cdots & 0 & 1    &  0 & -e_1 & 0 \\
		D_y & 1 & 0 & \cdots & 0 & 0    & 1  & -e_2 & 0.
	\end{array}		
\end{equation}
In the above $\Delta_B$ represents the toric polyhedron of the base which in our case is $\mathbb{P}^{n-1}$:
\begin{equation}
	\begin{array}{c|cccc}
		D_1     & 1 & 0 & \cdots & 0 \\ 
		D_2     & 0 & 1 & \cdots & 0 \\
		\vdots & \vdots & 0 & \ddots & 0 \\
		D_{n-1} & 0 & \cdots & 0 & 1 \\
		D_n & -1 & -1 & \cdots & -1
	\end{array}
\end{equation}
Furthermore, $e_1$ and $e_2$ are determined by the three types of elliptic curves which are realized as hypersurfaces:
\begin{equation}
	\left\{(e_1,e_2)\right\} = \left\{(-2,-3),(-1,-2),(-1,-1)\right\}.
\end{equation}
Using the Mori cone vectors $l^{(1)}$ and $l^{(2)}$ one derives (see \cite{Hosono:1993qy}) the 
Picard-Fuchs system $\PF_{\en}$ in \eqref{eq:pfoperator1} and \eqref{eq:pfoperator2}.
It depends on the Euler number of the base $\chi=\en$.
The vector $(e_1,e_2)$ determines $\PF_\en$ with:
\begin{eqnarray}
	(e_1,e_2) = (-2,-3) & \Rightarrow & (a_0,a_1,a_2) = (432,5/6,1/6) \nonumber \\
	(e_1,e_2) = (-1,-2) & \Rightarrow & (a_0,a_1,a_2) = (64,3/4,1/4) \nonumber \\
	(e_1,e_2) = (-1,-1) & \Rightarrow & (a_0,a_1,a_2) = (27,2/3,1/3).
\end{eqnarray}
We also include the last case where the fibre elliptic curve is realized as a complete intersection in $\mathbb{P}^3$ \cite{Klemm:2012sx}:
\begin{equation}
	(a_0,a_1,a_2) = (16,1/2,1/2).
\end{equation}

\section{Non-commutative rings}
Let $\C[z,\theta]=\C[z_1,z_2,\cdots, z_\hn,  \theta_1,\theta_2,\cdots,\theta_\hn]$ be a
non-commutative ring with non-commutative relations
$$
\theta_iz_i=z_i(\theta_i+1).
$$
Here, the variable $\theta_i:= z_i\frac{\partial}{\partial z_i}$ can be interpreted as  the logarithmic derivation. 
Let also $\PF$ be a finitely generated left ideal of $\C[z,\theta]$. For a fixed coordinate $z_2$, the restriction 
of $I$ to $z_2=0$ is defined to be
$$
\PF\mid_{z_2=0}:=\left \{ \ \ A\in \C[\hat z ,\hat \theta] \ \  
\mid \  \  \exists  B_1,B_2\in \C[z,\theta],\   A+z_2 B_1+\theta_2B_2 \in I \ \ \ \right\}.
$$
Here, $\hat z$ (resp. $\hat \theta$) is $z$  (resp. $\theta$) with $z_2$  (resp. $\theta_2$) removed. The computer algebra
{\sc Singular}, see \cite{GPS01}, has two libraries  nctools.lib, dmodapp.lib  for dealing with non-commutative ideals and their 
restrictions. If $\Pi^0$ is a holomorphic solution of $\L$ then $\Pi^0\mid_{z_2=0}$ is a holomorphic solution of $\L\mid_{z_2=0}$.  
We are mainly interested in the case where  $\hn=2$. In this paper we only need the following
\begin{prop}
 Let $\L_\en\subset \C[z_1,z_2,\theta_1,\theta_2]$ be the left ideal generated by $L_1$ and $L_2$ in 
 \eqref{eq:pfoperator1} and \eqref{eq:pfoperator2}. The restriction $\PF_{\en}\mid_{z_2=0}$ is generated by
 \begin{equation}
\label{khastambegam-2}
L:=\theta_1^2-a_0 z(\theta_1+a_1)(\theta_1+a_2). 
\end{equation}
\end{prop}
\begin{proof}
This follows immediately from the explicit form of $L_1$ and $L_2$.
\end{proof}
From now on we write $z,\theta$ etc. instead of $z_1,\theta_1$ in situations where we have taken the limit $z_2\to 0$.
In Appendix \ref{21/10/2015} we have computed more restrictions of non-commutative ideals.

\subsection{Modular forms and Gauss hypergeometric equation}
\label{1nov2015}
In the literature, we can find many examples of modular forms derived from the solutions of 
the Gauss hypergeometric equation \eqref{khastambegam-2} and for  particular values of $a_0,a_1,a_2$, however, a uniform approach 
for arbitrary parameters $a_i$ has been recently developed in \cite{hokh2} and \cite{hosseinkhosro}.
In \cite{GMCD-MQCY3} page 155 we have shown that the mirror map/Schwarz map of \eqref{khastambegam-2} 
has integral $q$-coefficients if and only if  the pair
$a_1,a_2$ belongs to the class of $28$ elements in Table \ref{28cases}.
{

\begin{table}
\centering
\begin{tabular}{|c|}
\hline
$
( 1/2 , 1/2 ),
( 2/3 , 1/3 ),
( 3/4 , 1/4 ),
( 5/6 , 1/6 ),
$
\\
$
( 1/6 , 1/6 ),
( 1/3 , 1/6 ),
( 1/2 , 1/6 ),
( 1/3 , 1/3 ),
( 2/3 , 2/3 ),$
\\
$( 1/4 , 1/4 ),
( 1/2 , 1/4 ),
( 3/4 , 1/2 ),
( 3/4 , 3/4 ),
( 1/2 , 1/3 ),
$
\\
$
( 2/3 , 1/6 ),
( 2/3 , 1/2 ),
( 5/6 , 1/3 ),
( 5/6 , 1/2 ),
( 5/6 , 2/3 ),
$
\\
$
( 5/6 , 5/6 ),
( 3/8 , 1/8 ),
( 5/8 , 1/8 ),
( 7/8 , 3/8 ),
( 7/8 , 5/8 ),
$
\\
$
( 5/12 , 1/12 ),
( 7/12 , 1/12 ),
( 11/12 , 5/12 ),
( 11/12 , 7/12 )
$\\ \hline 
 \end{tabular}
\caption{$N$-integral hypergeometric  mirror maps. }
\label{28cases}
\end{table}
}
For the proof of Theorem \ref{maintheo} we will need the condition $a_1+a_2=1$. This reduces our table above to the four cases of
$(a_1,a_2)$ shown in Table \ref{taghiirnadejanam}. The parameter $a_0$ is just a rescaling of $z_1$ and $\en$ can 
be any positive integer. For all $28$ examples in the Table \ref{28cases}  
one can determine an arithmetic group $\Gamma$, which is basically the monodromy
group of $L$, and the corresponding algebra of modular forms. 
For our purposes we only need the four cases relevant for this article and gathered in Table \ref{taghiirnadejanam}. 
In this table $E_i$'s and $\theta_i$'s are classical Eisenstein and theta series, respectively. The quasi-modular forms 
in each case are given by the $\C$-algebra generated by $E_2$ and the modular forms in the third column. In the last row note that
$\theta_4^4=\theta_3^4-\theta_2^4$. In the third row we have a polynomial relation between the three modular forms there, 
see for instance  the last section of \cite{ho14-II}.

\subsection{Hypergeometric functions}
\label{hypergeometricsection}
\def\losu {{G}}
In this section we first recall some well-known properties of the hypergeometric function 
$$
F(a,b|z)={_pF_q}(a_1,a_2,\cdots,a_p, b_1,b_2,\ldots b_q|z)=
\sum_{k=0}^\infty \frac{(a_1)_k\dots(a_p)_k}{(b_1)_k\cdots (b_q)_k k!} \, z^k, 
$$
$$
\quad |z|<1
,\ \ b_i\not =0, -1,-2,\cdots
$$ 
which satisfies the linear differential equation $L(a,b)F(a,b|z)=0$, where 
\begin{equation}
\label{3sept2015}
L(a,b)= \theta(\theta+b_1-1)(\theta+b_2-1)\cdots(\theta+b_q-1)-z(\theta+a_1)(\theta+a_2)\cdots (\theta+a_p)=0
\end{equation}
$(a_i)_k=a_i(a_i+1)(a_i+2)...(a_i+k-1),\,(a_i)_0 = 1$ is the Pochhammer symbol and $\theta=z\frac{d}{dz}$. 
For $q=p-1$ and $b_1=b_2=\cdots=b_q=1$, we have also the following logarithmic solution $\losu(a,1|z)+F(a,1|z)\log z$,  
where
\begin{equation}
\label{logsol}
\losu(a,1| z)=\sum_{k=1}^\infty \frac{(a_1)_k\cdots(a_p)_k}{(k!)^p}\big{[}\sum_{j=1}^ p\sum_{i=0}^{k-1}(\frac{1}{a_j+i}-\frac{1}{1+i})\big{]}z^k.
\end{equation}
We would like to find solutions of $L(a,b)$ when some of the $b_i$'s are  negative integers or zero.  
Let ${F}$ be any solution of
$L(a,b)$. We note that 
$z^a {F}$ satisfies
$$
(\theta-a)(\theta+b_1-1-a)(\theta+b_2-1-a)\cdots(\theta+b_q-1-a)-z(\theta+a_1-a)(\theta+a_2-a)\cdots (\theta+a_p-a)=0
$$
and so $z^{b_1-1} {F}$ satisfies 
$$
L(a_1-b_1+1,\cdots, a_p-b_1+1; 2-b_1, b_2-b_1+1,\ldots, b_q-b_1+1)=0.
$$
We will also need the following
\begin{equation}
 \frac{\partial }{\partial z}\ \ {F}(a_1\cdots,a_p,b_1,\cdots,b_q|z)=
 \frac{a_1a_2\cdots a_p}{b_1b_2\cdots b_q}\ \ {F}(a_1+1,\cdots,a_p+1, b_1+1,\cdots,b_q+1|z).
\end{equation}
Let us proceed to the discussion for the case of the classical Gauss hypergeometric equation with  $p=q+1=2$.  
We conclude that two solutions of 
$$
(\theta-n-1)\theta+z(\theta+a_1)(\theta+a_2)=0,\ \  n\in\N_0
$$
are given by $z^{n+1}F(a_1+n+1,a_2+n+1,n+2|z)$, where $F$ here refers to two solutions of $L(a_1+n+1,a_2+n+1,n+2)$. 
For the holomorphic solution $F$ this can be also seen using the limit
$$
\lim_{b_1\to -n} \frac{F(a_1,a_2,b_1\mid z)}{\Gamma(b_1)}= \frac{(a_1)_{n+1}(a_2)_{n+1}}{(n+1)!}z^{n+1}F(a_1+n+1,a_2+n+1;n+2|z).
$$

\section{Proof of Theorem \ref{maintheo}}
The proof of Theorem \ref{maintheo} is given at the level of periods or solutions of linear differential equations.
More precisely, we prove that for $f(z_1,z_2)\in \CYM_\en$ of the form
 $$
 f=f_0(z_1)+f_1(z_1)(q_2q_1^{\frac{\en}{2}})+\cdots +f_i(z_1)(q_2 q_1^{\frac{\en}{2}})^i+\cdots
 $$
 all $f_i(z)$ are in the field $\C(z,F,\theta F)$, where $F$ is the Gauss  hypergeometric function. After inserting the mirror map in
 $f_i$'s one gets the main result as stated in Theorem \ref{maintheo}, see \S\ref{1nov2015}. The above can equivalently be rewritten as
 $$
 f_i(z_1)=\left.\frac{1}{i!}(q_1^{\frac{\en}{2}}\frac{\partial}{\partial q_2})^{(i)}f\right|_{q_2=0}.
 $$
\subsection{Solutions of $\PF_\en$}
\label{solutionssection}
Let us consider the Picard Fuchs system. Let $L_1$ and $L_2$ be as in \eqref{eq:pfoperator1} and \eqref{eq:pfoperator2} 
and let $L$ be the Gauss hypergeometric equation \eqref{khastambegam-2}. It is also usefull to define
\begin{equation}
\label{harrishay}
L^{m}:= L-m\theta_1.
\end{equation}
We have $L_1=L-\en\theta_1\theta_2$. We have three solutions of $\PF_\en $ 
of the form:
\begin{eqnarray*}
 \Pi^0&=&1+\sum_{i=1}^\infty\Pi^0_i(z_1)z_2^i\\
  \Pi^{a}&=&\Pi^0\ln(z_a)+ \sum_{i=1}^\infty\Pi^a_i(z_1)z_2^i \quad a=1,2\\ 
\end{eqnarray*}
We need to analyze the following Wronskians in the limit $z_2=0$:
\begin{equation}
 W^{a,b}:=
 \det\begin{pmatrix}
      \Pi^0 & \theta_{a}\Pi^0\\
      \Pi^b & \theta_{a}\Pi^{1b}\\
     \end{pmatrix}, \ \ a,b=1,2. 
\end{equation}
All $W_{a,b}$'s satisfy Picard-Fuchs differential equations of higher orders.  We will write
\begin{equation}
 W^{a,b}=\sum_{i=0}^\infty W^{a,b}_i(z_1)z_2^i. 
\end{equation}
In what follows we will use
the derivation of differential operators with respect to the differentiation variable, for instance 
$$
\frac{\partial L^{m}}{\partial \theta}=2(1-a_0z)\theta-a_0z(a_1+a_2)-m. 
$$
\begin{prop}
We have
\begin{eqnarray}
 \label{8sept2015-1}
 L^{\en\cdot i}\Pi^0_i &=&0,\\  \label{8sept2015-2}
L^{\en\cdot i}(\Pi^0_i\log(z_1)+\Pi^1_i) &=& 0,\\  \label{8sept2015-4}
L^{\en\cdot i}\Pi^1_i&=& -\frac{L^{\en\cdot i}}{\partial \theta}\Pi^0_i, \\
\label{8sept2015-3}
L^{\en\cdot i}\Pi^2_i&=& \en\cdot \theta_1\Pi^0_i. 
\end{eqnarray}
\end{prop}
\begin{proof}
 We just apply the operator $L-\en \theta_1\theta_2$ to $\Pi^0, \Pi^2$ and $\Pi^1$, respectively, and we arrive at the above equalities. Note that  the third one is the reformulation
 of the second one using the equality
 \begin{equation}
 L(f\log z)=L(f)\log z+\frac{\partial L}{\partial \theta}(f).
\end{equation}
In general, for  two holomorphic functions $f$ and $g$ in $z$ and a differential operator $L$ of order $k$ 
 in $z,\theta$, 
 we have used 
 \begin{equation}
  L(fg)=L(f)\cdot g+\frac{\partial L}{\partial \theta}(f)\cdot \theta g+\cdots+ 
  \frac{\partial^k L}{\partial \theta^k}(f)\cdot \theta^{k} g.
 \end{equation}
We can verify this easily for $L=\theta^n$ by induction on $n$.

\end{proof}

\begin{prop}
 The Wronskian of the differential operator $L^{m}$ in \eqref{harrishay} (up to miltiplication with a constant) 
 is
 $$
 (1-a_0z)^{-a_1-a_2}(\frac{z}{1-a_0z})^m. 
 $$
\end{prop}
\begin{proof}
 We use the differential equation of the Wronskian $W$
 $$
\theta W=\frac{a_0(a_1+a_2)z+m}{1-a_0z}W.
 $$
\end{proof}
Using the properties of hypergemetric functions introduced in \S\ref{hypergeometricsection} 
we get the following: 
\begin{prop}
We have
 \begin{eqnarray}
 \label{nofa-1}
 \Pi^0_i &=& c^0_i 
 z^{\en\cdot i}\frac{\partial ^{\en\cdot i}}{\partial z^{\en\cdot i}}F(a_1, a_2,1|z),\\ \label{nofa-2}
 \Pi^0_i\log(z_1)+ \Pi^1_i &=& 
c^1_i z^{\en\cdot i}\frac{\partial ^{\en\cdot i}}{\partial z^{\en\cdot i}}\left(F(a_1, a_2,1|z)\ln(z_1)+\losu(a_1,a_2,1|z)\right)
+\tilde c^1_i\Pi^0_i,
 \end{eqnarray}
 where $c^0_i,c^1_i, \tilde c^1_i$ are constants. 
\end{prop}
The constants $c^0_i,c^i_1, \tilde c^1_i$ can be computed after applying the second operator $L_2$ to $\Pi^0,\Pi^1$. For the 
mathematical proof of Theorem \ref{maintheo} we do not need to compute them, however, for explicit verifications
of Theorem \ref{maintheo} one must compute them. 
From \eqref{nofa-1} it follows that $\Pi^0_i$ is in the field $\C(z,F,\theta F)$. Note that 
$$
\theta^2F=
\frac{a_0(a_1+a_2)z}{1-a_0z}\theta F+\frac{a_0a_1a_2z_1}{1-a_0z}F.
$$

\begin{prop}
\label{17oct2015-1}
 The quantities $W^{a,1}_i,\ \ a=1,2$ are in the field $\C(z, F, \theta F)$. 
\end{prop}
\begin{proof}
In \eqref{nofa-2} we use 
$$
F(a_1, a_2,1|z)\ln(z)+\losu(a_1,a_2,1|z)=F(a_1, a_2,1|z)\log(q)
$$
and 
\begin{equation}
\label{bony}
\frac{\partial \log(q)}{\partial z}=\frac{(1-a_0z)^{-a_1-a_2}z^{-1}}{F^2}
\end{equation}
and we write
$$
\Pi^0_i\log(z)+\Pi^1_i=\Pi^0_i\log(q)+A_i
$$
where $A_0=0$. We claim that $A_i$ is  in $\C(z,F,\theta F)$.  We can see this in two different ways. 
First, by using \eqref{nofa-2} and \eqref{bony},
second, by applying  the second differential
operator $L_2$ on $\Pi^0_i\log q+A_i$ which gives a  recursion for the $A_i$'s fixing them without ambiguity.
\end{proof}

\subsection{Nonhomogeneous differential equations}
We would like to solve the non-homogeneous equation 
\eqref{8sept2015-3}. In general, if we are given a second order linear differential operator $L=\theta^2+p(z)\theta+q(z)$ with two linearly independent solutions $y_1,y_2$, then a solution of the non-homogeneous differential equation $L=g(z)$ is given by 
$u_1y_1+u_2y_2$, where
$$
u_1=-\int \frac{y_2 g}{W(y_1,y_2)}dz,\ \ \ u_2:=\int \frac{y_1 g}{W(y_1,y_2)}dz
$$
and $W(y_1,y_2)=y_1\theta y_2-y_2\theta y_1=e^{-\int p(x)}$ is the Wronskian. We apply this to the non-homogeneous differential
equation \eqref{8sept2015-3} and obtain 
\begin{eqnarray*}
y_1(u_1+\frac{y_2}{y_1}u_2) &=& \en\cdot \Pi^0_i
\left ( 
-\int \tilde \Pi_i^1\theta\Pi_i^0(1-a_0z)^{a_1+a_2-1}(\frac{z}{1-a_0z})^{-\en\cdot i}\frac{dz}{z} \right. \\
& & \left.
+\frac{\tilde \Pi_i^1}{\Pi_i^0}\int \Pi_i^0\theta\Pi_i^0(1-a_0z)^{a_1+a_2-1}(\frac{z}{1-a_0z})^{-\en\cdot i}\frac{dz}{z}   
\right)
\end{eqnarray*}
where $\tilde \Pi^1_i$ is a second solution of $L^{\en i}=0$. Note that by \eqref{8sept2015-1} a first solution
is given by $\Pi^0_i$.
For 
\begin{equation}
a_1+a_2=1
\end{equation}
and $i=0$ we can solve these integrals and we get
$$
\Pi^2_0=-\frac{\en}{2}\Pi_0^0\log(\frac{1-a_0z}{z}). 
$$
This is defined up to addition of a linear combination of $\Pi^0_0$ and $\tilde \Pi_0^1$. We know that the 
original $\Pi^2_0$ arising from the solution $\Pi^2$ of $\PF_\en$ is holomorphic at $z_1=0$. 
Therefore, we add a multiple of $\tilde \Pi_0^1$ to the expression above and arrive at
\begin{equation}
\label{11sept2015}
\Pi^2_0=-\frac{\en}{2}\Pi_0^0\log(q\frac{1-a_0z}{z}),
\end{equation}
where $q=q_1\mid_{z_2=0}$. Note that for $i=0$, $\tilde \Pi^0_i$ is the logarithmic solution of the Gauss hypergeometric equation.   
We can add a multiple of $\Pi^0$ to $\Pi^2$ and assume 
that $\Pi^2_0$ is divisbale by $z$. 
In this way the formula of  $\Pi^2_0$  in \eqref{11sept2015} becomes unique. 
\begin{prop}
\label{17oct2015-2}
 The quantities $W^{a,2}_i,\ \ a=1,2$ are in the field $\C(z, F, \theta F)$. 
\end{prop}
\begin{proof}
Imitating the case of $\Pi^1_i$'s, we write
\begin{equation}
\label{dakhelhavapeima}
\Pi^0_i\log z_2+\Pi^2_i=\Pi^0_i \log\left( z_2q^{-\frac{\en}{2}} (\frac{1-a_0z_1}{z_1})^{{-\frac{\en}{2}} }        \right)+B_i
\end{equation}
After applying the second differential operator $L_2$ on the above expression 
we get a recursion for the $B_i$'s which shows that they are 
in the field $\C(z,F,\theta F)$. If we denote by $\tilde z_2$ the expression inside the logarithm in \eqref{dakhelhavapeima},
then using \eqref{bony} we have $\theta_1\log(\tilde z_2)\in \C(z,F,\theta F)$ and $\theta_2\log(\tilde z_2)=1$. 
Note that $B_0=0$ and hence it is in $\C(z, F,\theta F)$. This is the main reason for defining 
the logarithmic expression \eqref{dakhelhavapeima}. 
\end{proof}

\subsection{Differential field} 
\label{sec:difffield}
The field $\CYM_\en$ of Calabi-Yau modular forms defined in the introduction 
is by definition closed under derivations $\theta_1,\theta_2$. We have
$$
\begin{pmatrix}
 \frac{\partial z_1}{\partial \tau_1} &  \frac{\partial z_1}{\partial \tau_2}  \\
 \frac{\partial z_2}{\partial \tau_1} &  \frac{\partial z_2}{\partial \tau_2}  
\end{pmatrix}=
\frac{1}{\frac{\partial \tau_1}{\partial z_1}\frac{\partial \tau_2}{\partial z_2} - 
\frac{\partial \tau_1}{\partial z_2}\frac{\partial \tau_2}{\partial z_1} }
\begin{pmatrix}
 \frac{\partial \tau_2}{\partial z_2} &  -\frac{\partial \tau_1}{\partial z_2}  \\
 -\frac{\partial \tau_2}{\partial z_1} &  \frac{\partial \tau_1}{\partial z_1}  
\end{pmatrix}
$$
and therefore is invariant under 
\begin{eqnarray}
 \label{icerm-1}
\frac{\partial}{\partial \tau_2} &=& 
q_2\frac{\partial}{\partial q_2}=
\frac{(\Pi^0)^2}{W^{11}W^{22}-W^{21}W^{12}}
\left( 
-W^{21}\theta_1+W^{11}\theta_2
\right)\\
\label{icerm-2}
\frac{\partial}{\partial \tau_1} &=& q_1\frac{\partial}{\partial q_1}=
\frac{(\Pi^0)^2}{W^{11}W^{22}-W^{21}W^{12}}
\left( 
-W^{12}\theta_2+W^{22}\theta_1
\right)
\end{eqnarray}
This is still not enough to prove Theorem \ref{maintheo}. Proposition \ref{17oct2015-1} and
Proposition \ref{17oct2015-2} imply that the coefficents of the $z_2$-expansion
of elements of $\CYM_\en$ are in the field $\C(z,F,\theta F)$. 
\begin{prop}\label{11/09/2015}
We have 
\begin{equation}
 \label{NYtoRio}
 \left( \frac{1-a_0z_1}{z_1}\right )^{\frac{\en}{2}}  \frac{q_2 q_1^{\frac{\en}{2}}}{z_2}= 1+\sum_{i=1}^\infty C_iz_2^i \\
 \end{equation}
 and $C_i\in \C(z,F,\theta F)$.
\end{prop}
 Note that the quantity in \eqref{NYtoRio} does not belong to $\CYM_\en$, however, its $z_2$-expansion is similar to the $z_2$-expansion of the elements of $\CYM_\en$.
 \begin{proof}
 It follows from \eqref{11sept2015} that the quantity $X$ in \eqref{NYtoRio} starts with $1$. We have
 $$
 \partial_{2}X=X\partial_{2}\cdot \log(X)=X\cdot\left (\frac{W^{22}+\frac{\en}{2}W^{21} }{(\Pi^0)^2}-\frac{1}{z_2}\right )
 $$
 Substituting the left hand side of \eqref{NYtoRio} in the $X$ of the above equality we get a recursion of $C_i$'s
 which proves the Proposition.
 \end{proof}
 Becuase of Proposition \ref{11/09/2015}, it is natural to add the quantities 
 \begin{equation}
  \label{11S2015}
 \left( \frac{1-a_0z_1}{z_1}\right )^{\frac{\en}{2}},\ \ 
 \hbox{  and  } \frac{q_2 q_1^{\frac{\en}{2}}}{z_2}
 \end{equation}
 in \eqref{NYtoRio} to $\CYM_\en$ and 
 define $\check \CYM_\en$ to be the field generated by the elements of $\CYM_\en$ and \eqref{11S2015}. Note that
 for $\en$ even, the first element is already in $\CYM_\en$.
 \begin{prop}
 The field $\check \CYM_\en$ is invariant under the
 derivation $q_1^{-\frac{\en }{2}}\frac{\partial}{\partial q_2}$.
\end{prop}
The field $\CYM_\en$ is of course not invariant under  $\frac{\partial}{\partial q_i}$. It is 
invariant under the operator $q_2\frac{\partial}{\partial q_2}$, however, this operator cannot be used 
in order to compute the $q_2$-coeffiecients of an element in $\CYM_\en$.
\begin{proof}
 The proof follows from 
 $$
 q_1^{-\frac{\en }{2}}\frac{\partial}{\partial q_2}=
(\frac{q_2 q_1^{\frac{\en}{2}}}{z_2})^{-1}
 \frac{(\Pi^0)^2}{z_2(W^{11}W^{22}-W^{21}W^{12})}
\left( 
-W^{21}\theta_1+W^{11}\theta_2
\right)
 $$
Note that 
 \begin{eqnarray*}
 \left.
 \frac{(\Pi^0)^2W^{21}}{z_2(W^{11}W^{22}-W^{21}W^{12})}
 \right|_{z_2=0} & & 
\left.
 \frac{(\Pi^0)^2W^{11}}{(W^{11}W^{22}-W^{21}W^{12})}
 \right|_{z_2=0}
 \end{eqnarray*}
are in the field $\C(z,F,\theta F)$. 
\end{proof}

\section{Yukawa couplings for elliptically fibred Calabi-Yau fourfolds}
In this section we will focus on the class of elliptically fibred Calabi-Yau fourfolds. 
We will review mirror symmetry and proceed to compute 4-point functions which are also called Yukawa-couplings. 
Using the results from the previous sections we can express all Yukawa-couplings in terms of modular forms. 
This will provide the first example of a Calabi-Yau fourfold whose A-model correlation functions are expressed in 
terms of modular forms.
Here will will review how to compute periods of a Calabi-Yau fourfold $X$ and relate these to genus $0$ Gromov-Witten potentials of the mirror Calabi-Yau fourfold $\widetilde X$ where we will follow the references 
\cite{Greene:1993vm,Klemm:1996ts,Mayr:1996sh,Grimm:2009ef}. 

\subsection{$A$-side of the Mirror Symmetry}
\label{A-model}
In the case of fourfolds, in order to obtain zero virtual dimension for the moduli space of holomorphic maps, 
one needs to intersect the holomorphic curves with an extra four-cycle $\gamma$ in the Calabi-Yau $\widetilde X$. 
$\gamma$ can be any homology class Poincare dual to a cohomology class in the primary vertical subspace 
$H_V^{2,2}(\widetilde X)$. Here, for a Calabi-Yau $d$-fold $\widetilde X$,  $H_V^{k,k}(\widetilde X)$ consists of elements of the form
\begin{equation}
	\cO^{(k)}_{a} = 
	\sum_{i_1,\cdots,i_k} \alpha^{i_1,\cdots,i_k}_{a} J_{i_1} \wedge \ldots \wedge J_{i_k} \in H^{k,k}(\widetilde X),
\end{equation}
where $a=1,2\ldots$ enumerates a class of elements of $H_V^{k,k}(\widetilde X)$.
In the language of topological string theory the cohomology elements $\cO^{(k)}_{a}$ are also called degree 
$k$ A model operators. Among their non-zero correlation functions are the two-point functions
\begin{equation} \label{eq:metric}
	\eta_{ab}^{(k)} = \langle \mathcal{O}^{(k)}_{a} \mathcal{O}^{(d-k)}_{b}\rangle 
	= \int_{X} \mathcal{O}^{(k)}_{a} \wedge \mathcal{O}^{(d-k)}_{b},
\end{equation}
which do not receive any instanton corrections.
In mathematical terms, all these quantities are still integer valued and no $q$-expansion is attached.
However, the following three- and four-point functions do receive worldsheet instanton corrections
\begin{equation}
	C^{(1,1,2)}_{ab\gamma} = \langle \mathcal{O}^{(1)}_{a} \cO^{(1)}_{b} \cO^{(2)}_{\gamma}\rangle, \quad C^{(1,1,1,1)}_{abcd} =\langle \cO^{(1)}_a \cO^{(1)}_b \cO^{(1)}_c \cO^{(1)}_d\rangle,
\end{equation} 
and hence depend on the $q$-parameter. 
The genus $0$ Gromov-Witten potential is defined by
\begin{equation} \label{GWpot}
	F^0(\gamma) = \sum_{\beta \in H_2(X,\mathbb{Z})} N^0_{\beta}(\gamma) q^{\beta}, \quad \partial_{\tau_a} \partial_{\tau_b} F^0(\gamma) = C^{(1,1,2)}_{ab\gamma},
\end{equation}
where $N^0_{\beta}(\gamma)$ are the Gromov-Witten invariants which are in general rational and one has $q^{\beta} = \prod_{i=1}^{h^{1,1}} e^{2\pi i \tau_i \beta_i}$, see \cite{Klemm:2007in}. The potential (\ref{GWpot}) also admits an expansion in terms of integer invariants $n^0_{\beta}(\gamma) \in \mathbb{Z}$ as follows.
\begin{equation} \label{prepotential}
	F^0(\gamma) = \frac{1}{2} C^{0(1,1,2)}_{ab\gamma} \tau_a \tau_b + b^0_{a\gamma} \tau_a + a^0_{\gamma} + \sum_{\beta > 0} n^0_{\beta}(\gamma) \textrm{Li}_{2}(q^{\beta}),
\end{equation}
where we have
\begin{equation} \label{eq:3ptInt}
	C^{0(1,1,2)}_{ab\gamma} = \int_{\widetilde X} \cO^{(1)}_a \wedge \cO^{(1)}_b \wedge \cO^{(2)}_{\gamma}, 
\end{equation}
and 
\begin{equation}
\textrm{Li}_k(q) = \sum_{d=1}^{\infty} \frac{q^d}{d^k}.
\end{equation}

\subsection{$B$-side of the Mirror Symmetry}
Let us now come to the B model. Here the operators are elements of the horizontal subspace of the cohomology of the mirror Calabi-Yau variety $X$. 
Contrary to the three-fold case variations of the $(4,0)$ form $\Omega$ in the fourfold case do not span the full cohomology 
$H^4(X)$, but rather a subspace known as the horizontal subspace $H_H^4( X)$. 
By definition it is perpendicular to $H_V^4(X)$. It has the Hodge decomposition 
\begin{equation}
	H_H^4(X) = H^{4,0} \oplus H^{3,1} \oplus H^{2,2}_H \oplus H^{1,3} \oplus H^{0,4},
\end{equation}
where $H_H^{2,2}$ is the subspace of $H^{2,2}$ generated solely from the second variation of $\Omega$ with respect to the complex structure of $X$. Periods are then defined in terms of a basis $\gamma_a^{(i)}$ of $H^H_{4}(X)$ as follows
\begin{equation}
	\Pi^{(i) a} = \int_{\gamma_a^{(i)}} \Omega, \quad i=0,\ldots,4,
\end{equation}
where the cycles $\gamma^{(i)}_a$ are chosen such that they are dual to a basis $\hat \gamma_a^{(i)}$ of $H^{4-i,i}(X)$ with pairing
\begin{equation}
	\int_{\gamma_a^{(i)}} \hat \gamma_a^{(i)} = \delta^{ij} \delta_{ab}.
\end{equation}
Their $z$-expansion is of the form
\begin{eqnarray}
	\Pi^{(0)} & = &  1 + c_a z_a + \cO(\underline{z}^2), \nonumber \\
	\Pi^{(1) a} & = &  d_a(\underline{z}) + \log(z_a) \Pi^{(0)}(\underline{z}), \nonumber \\
	\Pi^{(2) \gamma} & = & \frac{1}{2} \sum_{a,b=1}^{h^{1,1}(X)} C^{0(1,1,2)}_{a b \gamma} 
	\left(d_a(\underline{z}) \log(z_b)  + d_b(\underline{z}) \log(z_a) + \Pi^{(0)}(\underline{z}) \log(z_a)\log(z_b)\right). \nonumber \\
	~ & ~ & + ~d^{\gamma}_{h^{1,1}(X)+1}, \label{eq:Bmodelperiods}
\end{eqnarray}
where the $d_a$ are polynomials of the form
\begin{eqnarray}
	d_1 & = & d_{1,a}^1 z_a + d^1_{2,a,b} z_a z_b + \cO(\underline{z}^3), \nonumber \\
	\vdots & ~ & ~ \nonumber \\
	d_{h^{1,1}(X)} & = & d_{1,a}^{h^{1,1}(X)} z_a + d_{2,a,b}^{h^{1,1}(X)} z_a z_b + \cO(\underline{z}^3), \nonumber \\
	d^{\gamma}_{h^{1,1}(X)+1} & = & 1 + \cO(\underline{z}). 
\end{eqnarray}
Furthermore, $C^{0(1,1,2)}_{a b \gamma}$ are constants defined in (\ref{eq:3ptInt}). Note that in section \ref{solutionssection} we have adopted the notation
$$
\Pi^{(0)}=\Pi^0,\ \Pi^{1(a)}=\Pi^a. 
$$
In order to introduce the prepotential (\ref{prepotential}) on the B model side we then have to use the following identities
\begin{equation} \label{mirrormap}
	F^0(\gamma) =\frac{\Pi^{(2) \gamma}}{\Pi^{(0)}}, \quad \tau_a = \frac{\Pi^{(1) a}}{\Pi^{(0)}}, ~a=1, \ldots, h^{1,1}(\widetilde X).
\end{equation}
This justifies the Ans\"atze (\ref{eq:Bmodelperiods}) for the B model periods.

\subsection{Yukawa couplings from Picard-Fuchs equations}
\label{sec:yuk}
Yukawa-couplings are defined through the holomorphic $(4,0)$-form $\Omega$ as follows:
\begin{equation} \label{eq:yukawa}
	C^{(1,1,1,1)}_{abcd} = \int_{X_z} \Omega \wedge \partial_a \partial_b \partial_c \partial_d \Omega,
\end{equation}
where $a,b,c,d \in \{1,\cdots,h^{3,1}(X)\}$ are complex structure moduli. 
We will utilize Griffirth transversality and the Picard-Fuchs equation to compute these four-point functions. 
Griffith transversality amounts to the following constraints:
\begin{eqnarray}
\label{eq:griffiths}
\int_{X_z} \Omega \wedge \partial_1^{i_1} \partial_2^{i_2}\cdots \partial_{h^{3,1}}^{i_{h^{3,1}}} \Omega & = & 0, \ \ \ \ 
i_1 + \cdots +i_{h^{3,1}} < 4
\end{eqnarray}
We will now present a formalism to compute four-point functions. In order to proceed we restrict our attention to 
Calabi-Yau fourfolds with a maximal number of $3$ complex structure moduli and define the functions
\begin{equation}
	W^{(i,j,k)} = \int_{X_z} \Omega \partial_1^i \partial_2^j \partial_3^k \Omega.
\end{equation}
Note that (\ref{eq:griffiths}) is equivalent to
\begin{eqnarray}
	W^{(i,j,k)} & = & 0  \quad \textrm{for}~ i+j+k<4, \nonumber \\
	W^{(i,j,k)} & = & C^{(1,1,1,1)}_{\underbrace{1\cdots 1}_{\textrm{$i$ times}}\underbrace{2\cdots 2}_{\textrm{$j$ times}}\underbrace{3\cdots 3}_{\textrm{$k$ times}}} ~\textrm{for}~ i+j+k = 4. \nonumber \\
\end{eqnarray}
Moreover, we arrive at further constraints by rewriting 
\begin{equation}
	\prod_{m=1}^3 \partial_m^{i_m} \int \Omega \wedge \prod_{m=1}^3 \partial_m^{j_m} \Omega = 0 \quad \textrm{for} ~ \sum_m i_m + j_m =5 \textrm{~and~} \sum_m j_m < 4,
\end{equation}
as first order differential equations in the four-point functions:
\begin{eqnarray}
	W^{(4,1,0)} & = & \frac{1}{2}(\partial_2 W^{(4,0,0)} + 4 \partial_1 W^{(3,1,0)}), \nonumber \\
	W^{(5,0,0)} & = & \frac{5}{2} \partial_1 W^{(4,0,0)}, \nonumber \\
	W^{(3,2,0)} &= & \frac{1}{2}(2 \partial_2 W^{(3,1,0)} + 3 \partial_1 W^{(2,2,0)}), \nonumber \\
	W^{(2,2,1)} & = & \frac{1}{2}(\partial_3 W^{(2,2,0)} + 2 \partial_2 W^{(2,1,1)} + 2 \partial_1 W^{(1,2,1)}), \nonumber \\
	W^{(3,1,1)} & = & \frac{1}{2}(\partial_3 W^{(3,1,0)} + \partial_2 W^{(3,0,1)} + 3 \partial_1 W^{(2,1,1)}),\nonumber \\ \label{eq:yukrel}
\end{eqnarray}
and all permutations of these.  
For the  differential operator $L_k= \sum_{\textbf{j}} f_k^{(\textbf{j})} \partial^{\textbf{j}}$ which annihilates $\Omega$ we 
have also
\begin{equation}
	\sum_{\textbf{j}} f_k^{{\textbf{j}}} W^{(\textbf{j})} = 0.
\end{equation}
This is obtained after taking the wedge product of the original equation with $\Omega$ and then integrating it over $X$. These equations can be supplemented further by applying more derivatives on the Picard-Fuchs operators so that one obtains algebraic equations relating the four-point functions. Acting with yet another derivative and using (\ref{eq:yukrel}) it is possible to obtain first order differential equations for the four-point functions which together with the algebraic constraints are enough to fix those up to a constant. The constant can then be fixed in terms of the classical intersection numbers of the mirror geometry as follows. Consider transforming the Yukawa-coupling to the mirror 
coordinates $\underline{\tau}$:
\begin{eqnarray}
	C^{(1,1,1,1)}_{abcd}(\underline{\tau}) & = & \sum_{e,f,g,h} 
	\frac{1}{\left(\Pi^{(0)}\right)^2} C^{(1,1,1,1)}_{efgh}(\underline{z}) 
	\frac{\partial z_e(\tau_i)}{\partial \tau_a} \frac{\partial z_f(\tau_i)}{\partial \tau_b}\frac{\partial z_g(\tau_i)}{\partial \tau_c}
	 \frac{\partial z_h(\tau_i)}{\partial \tau_d}  \nonumber \\
	~ & = & C^{0(1,1,1,1)}_{abcd} + \mathcal{O}(\tau_i), \label{eq:yukdef}
\end{eqnarray}
where the $C^{0(1,1,1,1)}_{abcd}$ are the classical intersection numbers of the mirror Calabi-Yau manifold. 
Using \eqref{icerm-1} and \eqref{icerm-2},  we can see that the Yukawa-couplings 
$C^{(1,1,1,1)}_{abcd}(\underline{\tau})$ are elements of the ring $\CYM_\en$. These couplings are related to the three-point functions through the identities:
\begin{equation} \label{eq:3pt4ptRel}
	C^{(1,1,1,1)}_{abcd}(\underline{\tau}) = C^{(1,1,2)}_{ab\gamma}(\underline{\tau}) \left({\eta^{(2)}}^{-1}\right)^{\gamma \delta}C^{(2,1,1)}_{\delta cd}(\underline{\tau}) .
\end{equation}

In the next section we will provide explicit examples for a particular family of Calabi-Yau fourfolds.

\section{Main example}
In this section we focus on the particular example of an elliptic fibration over $\mathbb{P}^3$ as also studied 
in \cite{Klemm:2007in}.

\subsection{Toric data}
The Mori cone vectors are given by
\begin{eqnarray}
	l^{(1)} & = &  (-6,0,0,0,0,2,3,1) \nonumber \\
	l^{(2)} & = & (0,1,1,1,1,0,0,-4). 
\end{eqnarray}
From these we deduce the Picard-Fuchs operators \eqref{eq:pfoperator1} and \eqref{eq:pfoperator2} 
with $\en=4, a_0=432, a_1=\frac{1}{6}, \ a_2=\frac{5}{6}$. 
In this example $H^{1,1}(\widetilde X)$ is generated by two elements $J_1$ and $J_2$ which are Poincar\'e dual 
to $D_1$ and $D_2$ introduced in the Introduction. We take the following linearly independent 
elements of $H^{2,2}_V$:
$$
\gamma_1:=J_2^2,\ \ \gamma_2:=\frac{1}{17}(4 J_1^2 + J_1 J_2)
$$
(In \cite{Klemm:2007in} we have also the notation $D_1=E$ and $D_2=B$, $E$ standing for the elliptic fibre and
$B$ standing for base).
The $A$-model notation for these objects that we used in \S\ref{A-model} is  $\gamma_i:= \cO^{(2)}_i,\ \ i=1,2$. 
The inverse of the intersection matrix
in this basis is 
$$
[\gamma_i\cdot\gamma_j]=(\eta^{(2)})^{-1}=
\begin{pmatrix}
 -4 & 1\\
 1 & 0
\end{pmatrix}
$$
Furthermore, we have
$$
\int_{\widetilde X}J_1^4=64,\ \ \int_{\widetilde X}J_1^3 J_2=16, \ \ \int_{\widetilde X}J_1^2 J_2^2=4, \ \ \int_{\widetilde X}J_1 J_2^3=1
$$
All other integrations of combniations of $J_i$'s over $\widetilde X$ are zero, see \cite{Klemm:2007in} for the details
of this computation. The BPS numbers for this particular Calabi-Yau manifold can be found in \cite{Klemm:2007in} and we also include them in the appendix of this paper. 

\subsection{Period expansions}
Next, we want to use the results of Proposition 1 - 8 to express the periods and Yukawa-couplings in terms of 
$SL(2,\mathbb{Z})$ modular forms.  In this case we have
\begin{equation} \label{eq:modrep}
	F(z) \rightarrow F(z(\tau)) = \left(E_4\right)^{\frac{1}{4}}, \quad z \rightarrow z(\tau) = \frac{1}{864} (1-\sqrt{1-1728/J}),
\end{equation}
$$
\theta F(z)\rightarrow \theta F(z(\tau)) = \frac{E_4^{1/4}(E_2 E_4-E_6)}{6(E_4^\frac{3}{2} + E_6)}.
$$
As a first step we solve for the constants of Proposition 4, we find:
\begin{eqnarray}
	c^0_0 = c^1_0 & = & 1 \nonumber \\
	c^0_1 = c^1_1 & = & 1 \nonumber \\
	c^0_2 = c^1_2 & = & \frac{1}{16} \nonumber \\
	c^0_3 = c^1_3 & = & \frac{1}{1296} \nonumber \\
	c^0_4 = c^1_4 & = & \frac{1}{331776} \nonumber \\
	c^0_5 = c^1_5 & = & \frac{1}{207360000} 
\end{eqnarray}
Regarding the constants $\tilde c^1_i$ we find that all of these are zero. Next, we compute the logarithmic periods and find that the quantities $A_i$ in Proposition 5 are given by:
\begin{eqnarray}
	A_0 & = & 0, \nonumber \\
	A_1 & = & -\frac{6(1-1688 z+ 1067904 z^2 - 307 556 352 z^3)}{(1-432 z)^4 F(z)}, \nonumber \\
	~     & \vdots & 
\end{eqnarray}
In particular, all $A_i \in \mathbb{Q}(z,F,\theta F)$ and have the form 
\begin{equation}
	A_i = \frac{P_i(z)}{(1-432 z)^{4 i} F(z)},
\end{equation}
where $P_i(z)$ are polynomials in $z$. For the $B_i$ which appear in Proposition 6 we find
\begin{equation}
	B_i = \frac{Q_i(z,F,\theta F)}{(1-432 z)^{4 i}F(z)},
\end{equation}
with polynomials $Q_i$. For example, we have
\begin{eqnarray}
	B_1(z) & = & \frac{1}{F(z) (1-432 z)^4} \times 4 \left(3 - 5064 z + 3203712 z^2 - 922669056 z^3 \right. \nonumber \\
	~          & ~ & + 3 F(z)^2 (1-1708 z + 1075344 z^2 - 291589632 z^3+62983360512 z^4) \nonumber \\
	~          & ~ & -\left. 5 F(z) \theta F(z) (1-2184 z + 1907712 z^2-828610560 z^3 + 143183904768 z^4)\right).
\end{eqnarray}
Regarding Proposition 7, we have 
\begin{equation}
	C_i(z) = \frac{R_i(z,F,\theta F)}{(1-432 z)^{4 i} F}.
\end{equation}
The $R_i$ are polynomials, the first of which is given by:
\begin{eqnarray}
	R_1(z) & = & 4 (3 F(z) (1-1708 z + 1075344 z^2 - 2915896332 z^3 + 62983360512 z^4) \nonumber \\
	~         &  ~ & - 5 \theta F (1-2184 z + 1907712 z^2 - 828610560 z^3 + 143183904768 z^4)).
\end{eqnarray}
In order to be able to apply the derivation defined in Proposition 8 we further need to compute 
\begin{eqnarray}
	\left.\frac{\left(\Pi^0\right)^2 W^{21}}{z_2 (W^{11} W^{22} - W^{21} W^{12})}\right|_{z_2=0} & = & -\frac{6(-1+1688 z - 1067904 z^2 + 3075563 z^3)}{(-1+432 z)^3}  \\
	\left.\frac{\left(\Pi^0\right)^2 W^{11}}{(W^{11} W^{22} - W^{21} W^{12})}\right|_{z_2=0} & = & 1.
\end{eqnarray}

\subsection{Yukawa couplings and modularity}

Using the above results together with Propositions 1 - 8 we can now express all 4-point functions defined in (\ref{eq:yukdef}) in terms of modular forms. In order to proceed we first write down the Yukawa-couplings on the B-model side as rational functions in the complex structure moduli:

\begin{eqnarray}
	W^{(4,0)} & = & -\frac{64}{z(1)^4 \Delta_1 } \nonumber \\
	W^{(3,1)} & = & \frac{16 (-1 + 432 z_1)}{z_1^3 z_2 \Delta_1} \nonumber \\
	W^{(2,2)} & = & -\frac{4 (1-432 z_1)^2}{z_1^2 z_2^2 \Delta_1} \nonumber \\
	W^{(1,3)} & = & \frac{(-1 + 432 z_1)^3}{z_1 z_2^2 \Delta_1} \nonumber \\
	W^{(0,4)} & = & \frac{64\left(-1 + 1728 z_1 - 1119744 z_1^2 + 322486272 z_1^3 \right)}{z_2^3 \Delta_1 \Delta_2}~, 
\end{eqnarray}
where $\Delta_1$, $\Delta_2$ are given by
\begin{eqnarray}
	\Delta_1 & = & -1 + 1728 z_1 - 1119744 z_1^2 + 322486272 z_1^3 + 34828517376 z_1^4 (-1 + 256 z_2), \nonumber \\
	\Delta_2 & = & -1 + 256 z_2.
\end{eqnarray}
We now want to compute $C^{(1,1,1,1)}_{abcd}(\underline{\tau})$ as an expansion in $q_2  = e^{-\tau_2}$ . Applying the derivation of Proposition 8 to (\ref{eq:yukdef}) we find after using (\ref{eq:modrep}):
\begin{eqnarray}
	C^{(1,1,1,1)}_{2222} & = & - q_2 \left(\frac{q_1^2}{\eta^{48}}\right) \left[\frac{5}{9} E_4 E_6(35 E_4^3+37 E_6^2)\right] \nonumber \\
	~ & ~ & - q_2^2 \left(\frac{q_1^4}{\eta^{96}}\right) \Big[\frac{5}{124416}E_4 E_6 (12377569 E_4^9 + 1960000 E_2 E_4^7 E_6 \nonumber \\
	~ & ~ & + 85433141 E_4^6 E_6^2 + 4144000 E_2 E_4^4 E_6^3 + 86392307 E_4^3 E_6^4 + 2190400 E_2 E_4 E_6^5 \nonumber \\
	~ & ~ & +11544823 E_6^6)\Big] + \mathcal{O}(q_2^3). \label{eq:yuk2222}
\end{eqnarray}
Notice that $C^{(1,1,1,1)}_{2222}$ is of modular weight $-2$\footnote{We assign weight $0$ to the combination $q_2 q_1^2$.} and if we define 
\begin{equation}
	Y^{(1)} = -\left(\frac{q_1^2}{\eta^{48}}\right) \frac{5}{9}E_4 E_6 (35 E_4^3+37 E_6^2),
\end{equation}
then equation (\ref{eq:yuk2222}) can be written as 
\begin{eqnarray}
	C^{(1,1,1,1)}_{2222} & = & q_2  Y^{(1)} - q_2^2 \Big[ \frac{5}{24} E_2 \left(Y^{(1)}\right)^2  \nonumber \\
	& ~ & + \left(\frac{q_1^4}{\eta^{96}}\right) \frac{5}{124416}E_4 E_6 (12377569 E_4^9+85433141 E_4^6 E_6^2+86392307 E_4^3 E_6^4 \nonumber \\
	~ & ~ & + 11544823 E_6^6)\Big] + ~\mathcal{O}(q_2^3)  \label{eq:yuk2222b}
\end{eqnarray}
This structure is reminiscent to the ``holomorphic anomaly" observed in \cite{Hosono:1999qc} in the case of elliptic Calabi-Yau threefolds and it would be very interesting to explore the significance of such an anomaly equation for the case of elliptic Calabi-Yau fourfolds further. In this paper we find evidence for such an anomaly structure also for the Gromov-Witten potential $F^0(\gamma_1)$ which we derive in the following.
Using the identity
\begin{equation}  \label{eq:YukToF}
	C^{(1,1,1,1)}_{2222} =-4C^{(1,1,2)}_{22\gamma_1}C^{(1,1,2)}_{22 \gamma_1}+
	C^{(1,1,2)}_{22\gamma_2}C^{(1,1,2)}_{22 \gamma_1}+C^{(1,1,2)}_{22\gamma_1}C^{(1,1,2)}_{22 \gamma_2}
\end{equation}
and the expansions
\begin{equation}
	C^{(1,1,2)}_{22\gamma_1} = 0 + \mathcal{O}(q_2), \quad C^{(1,1,2)}_{22\gamma_2} = 1 + \mathcal{O}(q_2),
\end{equation}
we derive
\begin{equation}
	F^0(\gamma_1)= - q_2 \left(\frac{q_1^2}{\eta^{48}}\right) \left[\frac{5}{18} E_4 E_6(35 E_4^3+37 E_6^2)\right] + \mathcal{O}(q_2^2).
\end{equation}
We observe that $F^0(\gamma_1)$ has modular weight $-2$. In order to derive the second order term $q_2$ we now impose an anomaly structure of the form
\begin{equation}
	F^0(\gamma_1) = q_2 \frac{1}{2} Y^{(1)} + q_2^2 \left[\left(\frac{q_1^4}{\eta^{96}}\right)P_{46}(E_4,E_6) + k E_2 \left(\frac{1}{2}Y^{(1)}\right)^2\right] + \mathcal{O}(q_2^3) ,
\end{equation}
where $P_{46}(E_4,E_6)$ is a polynomial of weight $46$ in $E_4$ and $E_6$ and $k$ is a constant. We find
\begin{eqnarray}
	k & = & -\frac{1}{12}, \nonumber \\
	P_{46}(E_4,E_6) & = & - \frac{5}{2985984} E_4 E_6 (29908007 E_4^9 + 207234483 E_4^6 E_6^2 + 208392741 E_4^3 E_6^4 \nonumber \\
	~ & ~ &  + 27245569 E_6^6). 
\end{eqnarray}
Using these results together with the identity (\ref{eq:YukToF}) we can solve for $F(\gamma_2)$ to first order in $q_2$:
\begin{eqnarray}
F^0(\gamma_2) & = & 1 + q_2 \left(\frac{q_1^2}{\eta^{48}}\right)\left[\frac{5}{10368}(10321 E_4^6 + 1680 (-24+E_2)E_4^4 E_6\right. \\
	~ & ~ &  + 59182 E_4^3 E_6^2 + 1776 (-24 + E_2)E_4 E_6^3 + 9985 E_6^4)\Big] + \mathcal{O}(q_2^2) \nonumber 
\end{eqnarray}
Note that $F^0(\gamma_2)$ is not a modular form of a definite weight but rather consits of pieces with weights $-2$ and $0$.

\appendix

\section{Table of BPS numbers for the main example}
\label{sec:invariants}
\begin{table}[here!]
\centering
\footnotesize{\begin{tabular} {|c|c|c|c|c|c|} 
\hline 
$d_1 \backslash d_2$ & 0 & 1 & 2 & 3 & 4 \\ \hline
0 & 0 & -20 & -820 & -68060 & -7486440 \\
 1 & 0 & 7680 & 491520 & 56256000 & 7943424000 \\
2 &  0 & -1800000 & -159801600 & -24602371200 & -4394584496640 \\
3 &  0 & 278394880 & 35703398400 & 7380433205760 & 1662353371955200 \\
4 &  0 & 623056099920 & -6039828417600 & -1683081588149760 & -478655396625235200 \\
5 &  0 & 97531011394560 & 2356890607411200 & 388243145737128960 & 119544387620870983680 \\ \hline
\end{tabular}\caption{$n^0_{d_1,d_2}(\gamma_1)$}} 
\end{table}

\begin{table}[here!]
\centering
\footnotesize{\begin{tabular} {|c|c|c|c|c|c|} 
\hline 
$d_1 \backslash d_2$ & 0 & 1 & 2 & 3 & 4 \\ \hline
0 & 0 & 0 & 0 & 0 & 0 \\
1 &  960 & 5760 & 181440 & 13791360 & 1458000000 \\
2 &  1920 & -1817280 & -98640000 & -10715760000 & -1476352644480 \\
3 &  2880 & 421685760 & 29972448000 & 4447212981120 & 783432258136320 \\
4 &  3840 & 2555202430080 & -6353500619520 & -1273702762398720 & -285239128072550400 \\
5 &  4800 & 506461104057600 & 4042353816604800 & 373520266906348800 & 86478430090747622400 \\
\hline
\end{tabular}\caption{$n^0_{d_1,d_2}(\gamma_2)$}} 
\end{table}

\section{More on Mirror Symmetry}
In this section we want to elaborate on details of period constructions on the B-side of the Mirror Symmetry and construct a more complete map between A-model and B-model quantities. We refer to the original references for a more thorough review.

We start by recalling that we can choose a dual basis $\hat \gamma_a^{(i)}$ of $H^{4-i,i}_H(X)$ (where $i=0,\ldots,4$ is the grading) with pairing
\begin{equation}
	\int_{\gamma_a^{(i)}} \hat{\gamma}_{b}^{(j)} = \delta^{ij} \delta_{ab}.
\end{equation}
The holomorphic four-form is then expanded as $\Omega=\sum_i \Pi^{(i) a} \hat{\gamma}_a^{(i)}$. Denoting the complex structure moduli space of $X$ by $\cM$ we find that for $z \in \cM$ the horizontal parts of $F^k = \oplus_{p=0}^k H^{4-p,p}(X_z)$ form holomorphic vector bundles for which one can introduce frames $\beta_a^{(k)}$ with the basis expansion
\begin{equation}
	\beta_a^{(k)} = \hat{\gamma}_a^{(k)} + \sum_{p>k} \Pi^{(p,k)~c}_a(z) \hat{\gamma}_c^{(p)}.
\end{equation} 
These $\beta_a^{(k)}$ are the basic operators of the B model and under mirror symmetry we have the exchange
\begin{equation}
	\cO^{(k)}_a \mapsto \left.\beta_a^{(k)}\right|_{z=0}.
\end{equation}
The depence of the $\Pi_a^{(p,k)}$ on $z$ is captured by the Picard-Fuchs operators $L_a(\underline{\theta},\underline{z})$. These are written in terms of the logarithmic derivatives $\theta_a = z_a \frac{\partial}{\partial z_a}$ with respect to the canonical complex variables $z_a$ defined at the large complex structure point. Define the formal limits
\begin{equation}
	L_i^{\textrm{lim}}(\underline{\theta}) = \textrm{lim}_{z_i \rightarrow 0} L_i(\underline{\theta},\underline{z}), i=1,\ldots,r,
\end{equation}
and consider the algebraic ring
\begin{equation}
	\cR = \mathbb{C}[\underline{\theta}]/(\cJ = \left\{L_1^{\textrm{lim}},\ldots,L_r^{\textrm{lim}}\right\}).
\end{equation}
One can define a grading for this ring by taking the ring at grade $k$, $\cR^{(k)}$ to be generated by a basis of degree $k$ polynomials whose number is given by $h^H_{4-k,k}(X)=h^V_{k,k}(\tilde X)$ for $k=0,\ldots,4$. There is a one-to-one map between the ring $\cR^{(k)}$  and solutions of the Picard-Fuchs equations at large radius.  A given ring element of the form $\cR^{(k) a} = \sum_{|\underline{\alpha}|=k} \frac{1}{(2\pi i)^k} m^a_{\underline{\alpha}} \theta^{\alpha_1}_{1} \cdots \theta^{\alpha_h}_h$ is mapped to a solution of the form
\begin{equation}
	\tilde{\Pi}^{(k) a} = X_0(\underline{z})\left[\mathbb{L}^{(k) a} + \cO(\log(z)^{|\alpha|-1}\right],
\end{equation}
where
\begin{equation}
	\mathbb{L}^{(k)~a} = \sum_{|\underline{\alpha}|=k} \frac{1}{(2\pi i)^k} \tilde{m}^a_{\underline{\alpha}}\log^{\alpha_1}(z_1)\ldots\log^{\alpha_h}(z_h),
\end{equation}
and $\tilde m^a_{\alpha}(\prod_i \alpha_i!) = m^a_{\alpha}$. Using the metric (\ref{eq:metric}) to move indices down we furthermore demand
\begin{equation}
	\cR^{(k)}_a \mathbb{L}^{(k) b} = \delta^b_a.
\end{equation}
With these definitions mirror symmetry, i.e. exchange of A and B model, is triggered by the identifications
\begin{equation}
	\theta_i \leftrightarrow J_i, \quad \left.\beta_a^{(k)}\right|_{z=0} = \left.\cR^{(k)}_a \Omega\right|_{z=0}.
\end{equation}

\section{Other families of elliptically fibred Calabi-Yau varieties}
\label{21/10/2015}
On our path to reformulate our main results for the Calabi-Yau $n$-folds with the Picard-Fuchs system 
\eqref{eq:pfoperator1} and \eqref{eq:pfoperator2}, we studied also many other elliptically fibred Calabi-Yau 
varieties and computed the corresponding Picard-Fuchs systems. For future investigation we have collected our computations in the
table bellow. In this table
$\mathbb{F}_i$'s are Hirzebruch surfaces.  The limit Picard-Fuchs equation in the variable $z_1$ means that the limit
is taken with respect to all other variables except $z_1$. 


\begin{sidewaystable}[!htbp]
\centering
\begin{tabular}{|c|c|c|c|c|c|}
\hline
No. & CY & Base & Fibre & PF-system & $\lim_{z_i\rightarrow 0}$ PF-system\\ \hline

0
& \multirow{4}{*}{3-fold}  & \multirow{2}{*}{ $\mathbb{P}^2$ } & \multirow{2}{*}{Elliptic}&
 $ L_1=\theta_1(\theta_1-3\theta_2)-12z_1(6\theta_1+1)(6\theta_1+5)$ 
 &
\href{http://w3.impa.br/~hossein/WikiHossein/files/Singular%20Codes/2015-06-PFSystemAndLimits-EmanuelMurad.txt}
{$\mathcal{L}_{z_1=0} = \theta_2^3+3z_2\theta_2(3\theta_2+1)(3\theta_2+2)$}
 \\ 
& & & & 
  $ L_2=\theta_2^3+z_2(3\theta_2-\theta_1+0)(3\theta_2-\theta_1+1)(3\theta_2-\theta_1+2)$ &    
 
$\mathcal{L}_{z_2=0} = \theta_1^2-432z_1(\theta_1+\frac{1}{6})(\theta_1+\frac{5}{6})$
\\   \cline{3-6}

1
&  & \multirow{2}{*}{$\mathbb{P}^1$} & \multirow{2}{*}{K3 ($d=4$) }&
 ${L}_1 = \theta_1^2 (\theta_1 - 2 \theta_2) - 4 z_1 (4 \theta_1 + 3) (4 \theta_1 + 2) (4 \theta_1 + 1)$  &
\href{http://w3.impa.br/~hossein/WikiHossein/files/Singular%20Codes/2015-04-PFSystem-CY3CY4.txt}
{$ \theta_1^3  -4z_1(4\theta_1+1)(4\theta_1+2)(4\theta_1+3)$} \\ 
& & & & 
 ${L}_2 = \theta_2^2 - z_2 (2 \theta_2 - \theta_1 + 1) (2 \theta_2 - \theta_1)$    &  
\href{http://w3.impa.br/~hossein/WikiHossein/files/Singular%20Codes/2015-04-PFSystem-CY3CY4.txt}
{$\theta_2^2-2z_2\theta_2(2\theta_2+1)$}  
 \\   \hline
 2
& \multirow{17}{*}{4-fold} & \multirow{2}{*}{$\mathbb{P}^3$} & \multirow{2}{*}{Elliptic}  &
$L_1=\theta_1 (\theta_1-4\theta_2) - 12 z_1 (6 \theta_1+5)(6\theta_1+1)$  & 
$\theta_1^2  - 12 z_1 (6\theta_1+5)(6\theta_1+1)$ \\
 & & & & $L_2=\theta_2^4 - z_2 (4\theta_2-\theta_1)(4\theta_2-\theta_1+1)(4\theta_2-\theta_1+2)(4\theta_2-\theta_1+3)$ & 
$\theta_2^4-4z_2\theta_2(4\theta_2+1)(4\theta_2+2)(4\theta_2+3)$.
 \\ \cline{3-6}

3
&  & \multirow{2}{*}{$\mathbb{F}_0$} & \multirow{9}{*}{K3 ($d=4$)}  &
${L}_1=\theta_1^2(\theta_1-2\theta_2-2\theta_3)- 8 z_1 (1+2 \theta_1)(1+4\theta_1)(3+4\theta_1)$  & 
\href{http://w3.impa.br/%7Ehossein/WikiHossein/files/Singular%20Codes/2015-04-PFSystemHirzebruch.txt}
{$\theta_1^3-4z_1(4\theta_1+1)(4\theta_1+2)(4\theta_1+3)$} \\
 & & & & ${L}_{2,0}=\theta_2^2 - z_2 (-1+\theta_1-2\theta_2-2\theta_3)(\theta_1-2\theta_2-2\theta_3)$ & 
\href{http://w3.impa.br/%7Ehossein/WikiHossein/files/Singular%20Codes/2015-05-PFSystemHirzebruch-z1=z2=0.txt}
{$\theta_2^2-4z_2\theta_2(\theta_2+\frac{1}{2})$}
 \\

& & & & ${L}_{3}=\theta_3^2 - z_3 (-1+\theta_1-2\theta_2-2\theta_3)(\theta_1-2\theta_2-2\theta_3)$ & 
\href{http://w3.impa.br/%7Ehossein/WikiHossein/files/Singular%20Codes/2015-05-PFSystemHirzebruch-z1=z2=0.txt}
{$\theta_3^3-4z_3\theta_3(\theta_3+\frac{1}{2})$}
 \\ \cline{5-6}
4 & & \multirow{2}{*}{$\mathbb{F}_1$} & & ${L}_1=\theta_1^2(\theta_1-2\theta_2-2\theta_3)- 8 z_1 (1+2 \theta_1)(1+4\theta_1)(3+4\theta_1)$ &
\href{http://w3.impa.br/%7Ehossein/WikiHossein/files/Singular%20Codes/2015-04-PFSystemHirzebruch.txt}
{
$\theta_1^3-4z_1(4\theta_1+1)(4\theta_1+2)(4\theta_1+3)$
}
\\
& & & & ${L}_{2,1}=\theta_2^2+z_2(\theta_1-\theta_2-2\theta_3)(\theta_2-\theta_3)$ & 
\href{http://w3.impa.br/%7Ehossein/WikiHossein/files/Singular%20Codes/2015-05-PFSystemHirzebruch-z1=z2=0.txt}
{
$\theta_2^2-2z_2\theta_2^2$
}
\\  
& & & & ${L}_{3}=\theta_3^2 - z_3 (-1+\theta_1-2\theta_2-2\theta_3)(\theta_1-2\theta_2-2\theta_3)$ & 
\href{http://w3.impa.br/%7Ehossein/WikiHossein/files/Singular%20Codes/2015-05-PFSystemHirzebruch-z1=z2=0.txt}
{
$\theta_3^2-2z_3\theta_3(2\theta_3+1)$
}
 \\   \cline{5-6} 
5 & & \multirow{3}{*}{$\mathbb{F}_2$} &  &
${L}_1=\theta_1^2(\theta_1-2\theta_2-2\theta_3)- 8 z_1 (1+2 \theta_1)(1+4\theta_1)(3+4\theta_1)$   & 
\href{http://w3.impa.br/%7Ehossein/WikiHossein/files/Singular%20Codes/2015-04-PFSystemHirzebruch.txt}
{$\theta_1^3-4z_1(4\theta_1+1)(4\theta_1+2)(4\theta_1+3)$}
 \\
& & & & ${L}_{2,2}=\theta_2^2 - z_2 (2\theta_2-\theta_3)(1+2\theta_2 - \theta_3)$ & 
\href{http://w3.impa.br/%7Ehossein/WikiHossein/files/Singular%20Codes/2015-05-PFSystemHirzebruch-z1=z2=0.txt}
{$\theta_2^2-4z_2\theta_2(\theta_2+\frac{1}{2})$}
 \\ 
& & & & ${L}_3=\theta_3^2 - z_3 (-1+\theta_1-2\theta_2-2\theta_3)(\theta_1-2\theta_2-2\theta_3)$ & 
\href{http://w3.impa.br/%7Ehossein/WikiHossein/files/Singular%20Codes/2015-04-PFSystemHirzebruch.txt}
{
$\theta_3^2-2z_3\theta_3(2\theta_3+1)$
}
\\ \cline{3-6}
6
 & & \multirow{6}{*}{$\mathbb{P}^2$} & \multirow{2}{*}{K3 ($d=4$)} &
${L}_1 = \theta_1^2 (\theta_1 - 3\theta_2) -8 z_1 (1+2\theta_1)(1+4\theta_1)(3+4\theta_1) $  &  
 \href{http://w3.impa.br/%7Ehossein/WikiHossein/files/Singular%20Codes/2015-04-PFSystemAndLimits-z2=0.txt}{
$ \theta_1^3  -4z_1(4\theta_1+1)(4\theta_1+2)(4\theta_1+3)$}  \\ 
& & & & ${L}_2 = \theta_2^3 - z_2 (-2 + \theta_1 - 3 \theta_2)(-1+\theta_1-3\theta_2)(\theta_1-3\theta_2)$ & 
\href{http://w3.impa.br/%7Ehossein/WikiHossein/files/Singular%20Codes/2015-06-PFSystemAndLimits-z1=0.txt}
{$\theta_2^3+3z_2\theta_2(3\theta_2+1)(3\theta_2+2)$}
 \\ \cline{4-6}
7 & & & \multirow{2}{*}{K3 ($d=6$)} &
$ L_1=\theta_1^2(\theta_1 - 3 \theta_2) + 6 z_1 (1+2 \theta_1)(1+3 \theta_1)(2+3\theta_1)$ &      
$\theta_1^3+6z_1(3\theta_1+1)(3\theta_1+2)(2\theta_1+1)$        \\ 
& & & & $ L_2=\theta_2^3 - z_2 (-2+\theta_1 - 3 \theta_2)(-1+\theta_1-3\theta_2)(\theta_1-3\theta_2)$ & 
$\theta_2^3+3z_2\theta_2(3\theta_2+1)(3\theta_2+2)$
 \\  \cline{4-6}
8 & & & \multirow{2}{*}{K3 ($d=8$)} &
$ L_1=\theta_1^2(\theta_1 - 3\theta_2) - 8 z_1 (1+2\theta_1)^3$ &    $\theta_1^3-8z_1(2\theta_1+1)^3$\\ 
& & & & 
$ L_2=\theta_2^3 - z_2 (-2 + \theta_1 - 3\theta_2)(-1+\theta_1 - 3 \theta_2)(\theta_1 - 3\theta_2)$ & 
$\theta_2^3+3z_2\theta_2(3\theta_2+1)(3\theta_2+2)$
 \\ \hline

\end{tabular}\caption{PF-system} 
\label{PF-system}
\end{sidewaystable}

\newcommand{\etalchar}[1]{$^{#1}$}
\def\cprime{$'$} \def\cprime{$'$} \def\cprime{$'$} \def\cprime{$'$}


\end{document}

%% file: notations.tex
\def\Z{\mathbb{Z}}                   
\def\Q{\mathbb{Q}}                   
\def\C{\mathbb{C}}                   
\def\N{\mathbb{N}}                   
\def\uhp{{\mathbb H}}                
\def\A{\mathbb{A}}                   
\def\dR{{\rm dR}}                    
\def\F{{\cal F}}                     
\def\Sp{{\rm Sp}}                    
\def\Gm{\mathbb{G}_m}                 
\def\Ga{\mathbb{G}_a}                 
\def\Tr{{\rm Tr}}                      
\def\tr{{{\mathsf t}{\mathsf r}}}                 
\def\spec{{\rm Spec}}            
\def\ker{{\rm ker}}              
\def\GL{{\rm GL}}                
\def\k{{\sf k}}                     
\def\ring{{\sf R}}                   
\def\X{{\sf X}}                      
\def\T{{\sf T}}                      
\def\Ts{{\sf S}}
\def\cmv{{\sf M}}                    
\def\BG{{\sf G}}                       
\def\podu{{\sf pd}}                   
\def\ped{{\sf U}}                    
\def\per{{\sf  P}}                   
\def\gm{{\sf  A}}                    
\def\gma{{\sf  B}}                   
\def\ben{{\sf b}}                    

\def\Rav{{\mathfrak M }}                     
\def\Ram{{\mathfrak C}}                     
\def\Rap{{\mathfrak G}}                     

\def\nov{{\sf  n}}                    
\def\mov{{\sf  m}}                    
\def\Yuk{{\sf C}}                     
\def\Ra{{\sf R}}                      
\def\hn{{\sf h}}                      
\def\cpe{{\sf C}}                     
\def\g{{\sf g}}                       
\def\t{{\sf t}}                       
\def\pedo{{\sf  \Pi}}                  

\def\Der{{\rm Der}}                   
\def\MMF{{\sf MF}}                    
\def\codim{{\rm codim}}                
\def\dim{{\rm    dim}}                
\def\Lie{{\rm Lie}}                   
\def\gg{{\mathfrak g}}                

\def\u{{\sf u}}                       

\def\imh{{  \Psi}}                 
\def\imc{{  \Phi }}                  
\def\stab{{\rm Stab }}               
\def\Vec{{\rm Vec}}                 
\def\prim{{\rm prim}}                  

\def\Fg{{\sf F}}     
\def\hol{{\rm hol}}  
\def\non{{\rm non}}  
\def\alg{{\rm alg}}  

\def\bcov{{\rm \O_\T}}       

\def\leaves{{\cal L}}        

\def\GM{{\rm GM}}

\def\perr{{\sf q}}        
\def\perdo{{\cal K}}   
\def\sfl{{\mathrm F}} 
\def\sp{{\mathbb S}}  

\newcommand\diff[1]{\frac{d #1}{dz}} 
\def\End{{\rm End}}              

\def\sing{{\rm Sing}}            
\def\cha{{\rm char}}             
\def\Gal{{\rm Gal}}              
\def\jacob{{\rm jacob}}          
\def\tjurina{{\rm tjurina}}      
\newcommand\Pn[1]{\mathbb{P}^{#1}}   
\def\Ff{\mathbb{F}}                  

\def\O{{\cal O}}                     
\def\as{\mathbb{U}}                  
\def\ring{{\mathsf R}}                         
\def\R{\mathbb{R}}                   

\newcommand\ep[1]{e^{\frac{2\pi i}{#1}}}
\newcommand\HH[2]{H^{#2}(#1)}        
\def\Mat{{\rm Mat}}              
\newcommand{\mat}[4]{
     \begin{pmatrix}
            #1 & #2 \\
            #3 & #4
       \end{pmatrix}
    }                                
\newcommand{\matt}[2]{
     \begin{pmatrix}                 
            #1   \\
            #2
       \end{pmatrix}
    }
\def\cl{{\rm cl}}                

\def\hc{{\mathsf H}}                 
\def\Hb{{\cal H}}                    
\def\pese{{\sf P}}                  

\def\PP{\tilde{\cal P}}              
\def\K{{\mathbb K}}                  

\def\M{{\cal M}}
\def\RR{{\cal R}}
\newcommand\Hi[1]{\mathbb{P}^{#1}_\infty}
\def\pt{\mathbb{C}[t]}               
\def\W{{\cal W}}                     
\def\gr{{\rm Gr}}                
\def\Im{{\rm Im}}                
\def\Re{{\rm Re}}                
\def\depth{{\rm depth}}
\newcommand\SL[2]{{\rm SL}(#1, #2)}    
\newcommand\PSL[2]{{\rm PSL}(#1, #2)}  
\def\Resi{{\rm Resi}}              

\def\L{{\cal L}}                     
\def\Aut{{\rm Aut}}              
\def\any{R}                          
\newcommand\ovl[1]{\overline{#1}}    

\newcommand\mf[2]{{M}^{#1}_{#2}}     
\newcommand\mfn[2]{{\tilde M}^{#1}_{#2}}     

\newcommand\bn[2]{\binom{#1}{#2}}    
\def\ja{{\rm j}}                 
\def\Sc{\mathsf{S}}                  
\newcommand\es[1]{g_{#1}}            
\newcommand\V{{\mathsf V}}           
\newcommand\WW{{\mathsf W}}          
\newcommand\Ss{{\cal O}}             
\def\rank{{\rm rank}}                
\def\Dif{{\cal D}}                   
\def\gcd{{\rm gcd}}                  
\def\zedi{{\rm ZD}}                  
\def\BM{{\mathsf H}}                 
\def\plf{{\sf pl}}                             
\def\sgn{{\rm sgn}}                      
\def\diag{{\rm diag}}                   
\def\hodge{{\rm Hodge}}
\def\HF{{\sf F}}                                
\def\WF{{\sf W}}                               
\def\HV{{\sf HV}}                                
\def\pol{{\rm pole}}                               
\def\bafi{{\sf r}}
\def\id{{\rm id}}                               
\def\gms{{\sf M}}                           
\def\Iso{{\rm Iso}}                           

\def\hl{{\rm L}}    
\def\imF{{\rm F}}
\def\imG{{\rm G}}